\documentclass[reqno]{amsart}
\usepackage{mathrsfs}
\usepackage{color}
\usepackage{amsmath}
\usepackage{amsfonts}
\usepackage{amssymb}
\usepackage{graphicx}%
\usepackage{hyperref}

 \newtheorem{Theorem}{Theorem}[section]
 
 \newtheorem{Lemma}[Theorem]{Lemma}
 \newtheorem{Proposition}[Theorem]{Proposition}

 \newtheorem{Remark}[Theorem]{Remark}

 \numberwithin{equation}{section}
\begin{document}

\title[strong openness conjecture]
 {strong openness conjecture for plurisubharmonic functions}

\author{Qi'an Guan}
\address{Qi'an Guan: School of Mathematical Sciences,
Peking University, Beijing, 100871, China.}
\email{guanqian@amss.ac.cn}
\author{Xiangyu Zhou}
\address{Xiangyu Zhou: Institute of Mathematics, AMSS, and Hua Loo-Keng Key Laboratory of Mathematics, Chinese Academy of Sciences, Beijing, China}
\email{xyzhou@math.ac.cn}

\thanks{The second author was partially supported by NSFC}



\date{2013.10.06}

\dedicatory{}

\commby{}

%%% ----------------------------------------------------------------------

\begin{abstract}
In this article, we give a proof of the strong openness conjecture
for plurisubharmonic functions posed by Demailly.
\end{abstract}

%%% ----------------------------------------------------------------------
\maketitle
%%% ----------------------------------------------------------------------

\section{Introduction}

Let $X$ be a complex manifold with dimension $n$, and $\varphi$ be a
plurisubharmonic function on $X$. Let $\mathcal{I}(\varphi)$ be the
multiplier ideal sheaf associated to the plurisubharmonic function
$\varphi$ on $X$. Denote by
$$\mathcal{I}_{+}(\varphi):=\cup_{\varepsilon>0}\mathcal{I}((1+\varepsilon)\varphi).$$

In \cite{berndtsson13}, Berndtsson gave a proof of the openness
conjecture of Demailly and Koll\'{a}r in \cite{D-K01}:

\emph{\textbf{Openness conjecture:} Assume that
$\mathcal{I}(\varphi)=\mathcal{O}_{X}$. Then
$$\mathcal{I}_{+}(\varphi)=\mathcal{I}(\varphi).$$}

For the dimension two case it was proved by Favre and Jonsson in
\cite{FM05j} (see also \cite{FM05v}). For arbitrary dimension it has
been reduced to a purely algebraic statement in \cite{JM13}.

The strong openness conjecture is stated as follows, which is an
open question posed by Demailly in \cite{demailly-note2000},
\cite{demailly2010} (see also \cite{D-P03}, \cite{LiPhD},
\cite{JM12}, \cite{Gue12}, \cite{Cao12}, \cite{JM13}, \cite{Li2013},
\cite{Mat2013}, etc. ):

\emph{\textbf{Strong openness conjecture:} For any plurisubharmonic function $\varphi$ on $X$,
we have $$\mathcal{I}_{+}(\varphi)=\mathcal{I}(\varphi).$$}

The strong openness conjecture implies the openness conjecture.

In the present paper,
we prove the strong openness conjecture by the following theorem:

\begin{Theorem}
\label{t:strong_open}(solution of the strong openness conjecture)
Let $\varphi$ be a negative plurisubharmonic function on the unit
polydisc $\Delta^{n}\subset\mathbb{C}^{n}$, which satisfies
$$\int_{\Delta^{n}}|F|^{2}e^{-\varphi}d\lambda_{n}<+\infty,$$
where $F$ is a holomorphic function on $\Delta^{n}$ and
$\lambda_{n}$ is the Lebesgue on $\mathbb{C}^{n}$. Then there exist
$r>0$ and $p>1$, such that
$$\int_{\Delta^{n}_{r}}|F|^{2}e^{-p\varphi}d\lambda_{n}<+\infty,$$
where $r<1$.
\end{Theorem}

For $n\leq 2$, the above theorem was proved in \cite{JM12} by
studying the asymptotic jumping numbers for graded sequences of
ideals.

There are some immediate corollaries of the strong openness
conjecture, we omit them here. Some further results are being in
preparation (see our another paper), including a conjecture of
Demailly and Koll\'ar in \cite{D-K01} about the growth of the
measure of the sublevel set of psh functions related to the complex
singularity exponents.

\section{Some results used in the proof of the theorem}

In this section, we will show some preliminary results used in the
proof of the main theorem.

\subsection{A property of $L^{1}$ integrable function}
$\\$

Let $G$ be a positive Lebesgue measurable and integrable function on
a domain $\Omega\subset\subset\mathbb{C}^{n}$, i.e.,
$$\int_{\Omega}Gd\lambda_{n}<+\infty.$$

Consider the function
$$F_{G}(t):=\sup\{a|\lambda_{n}(\{G\geq a\})\geq t\},$$
$t\in(0,\lambda_{n}(\Omega)].$

We first discuss the finiteness of $F_{G}(t)$:

If for some $t_{0}$, $F_{G}(t_{0})=+\infty$, then $\lambda_{n}(G\geq
A_{j})\geq t_{0}$, where $A_{j}$ is a number sequence tending to
$+\infty$ when $j\to+\infty$. Since $G$ is $L^1$ integrable, we have
$$t_{0}A_{j}\leq A_{j}\lambda_{n}(\{G\geq A_{j})\leq\int_{\{G\geq A_{j}\}}Gd\lambda_{n}\leq\int_{\Omega}Gd\lambda_{n}<+\infty.$$
Letting $A_{j}\to+\infty$, we thus obtain a contradiction. Therefore
$F_{G}(t)<+\infty$ for any $t$.

Secondly, we discuss the decreasing property of $F_{G}(t)$:

Note that $\{A|\lambda_{n}(G\geq A)\geq
t_{1}\}\supset\{A|\lambda_{n}(G\geq A)\geq t_{2}\}$, where
$t_{1}\leq t_{2}$. Then we have $F_{G}(t_{1})\geq F_{G}(t_{2})$,
when $t_{1}\leq t_{2}$.

The first lemma is about the sublevel set of $F_{G}$:

\begin{Lemma}
\label{l:level}
We have
\begin{equation}
\label{equ:level}
\mu_{\mathbb{R}}(\{t|F_{G}(t)\geq a\})=\lambda_{n}(\{G\geq a\}),
\end{equation}
for any $a>0$,
where $\mu_{\mathbb{R}}$ is the Lebesgue measure on $\mathbb{R}$.
Moreover, we have
$$\mu_{\mathbb{R}}(\{t|F_{G}(t)> a\})=\lambda_{n}(\{G> a\}).$$
\end{Lemma}

\begin{proof}
Since
$$\mu_{\mathbb{R}}(\{t|F_{G}(t)> a\})=\lim_{k\to+\infty}\mu_{\mathbb{R}}(\{t|F_{G}(t)\geq a+\frac{1}{k}\}),$$
and
$$\lambda_{n}(\{G> a\})=\lim_{k\to+\infty}\lambda_{n}(\{G\geq a+\frac{1}{k}\}),$$
we only need to prove
$$\mu_{\mathbb{R}}(\{t|F_{G}(t)\geq a\})=\lambda_{n}(\{G\geq a\}),$$
for any $a>0$.

Note that $$\sup\{a_{1}|\lambda_{n}(\{G\geq
a_{1}\})\geq\lambda_{n}(\{G\geq a\})\}\geq a.$$ Then we have
$$\lambda_{n}(\{G\geq a\})\in\{t|\sup\{a_{1}|\lambda_{n}(\{G\geq a_{1}\})\geq t\}\geq a\},$$
therefore
$$\{t|\sup\{a_{1}|\lambda_{n}(\{G\geq a_{1}\})\geq t\}\geq a\}\supseteq\{t|\lambda_{n}(\{G\geq a\})\geq t\},$$
where $t>0$.

If $"\supseteq"$ in the above relation is strictly $"\supset"$, then
there exists $t_{0}$ such that

1). $\sup\{a_{1}|\lambda_{n}(\{G\geq a_{1}\})\geq t_{0}\}\geq a$;

2). $t_{0}>\lambda_{n}(\{G\geq a\})$.

Let
$$a_{0}:=\sup\{a_{1}|\lambda_{n}(\{G\geq a_{1}\})\geq t_{0}\}\geq a.$$
Note that
$$\lambda_{n}(\cap_{a_{1}<a_{0}}\{G\geq a_{1}\})=\inf_{a_{1}<a_{0}}\lambda_{n}(\{G\geq a_{1}\}).$$
Then we have $\lambda_{n}(\{G\geq a_{0}\})\geq t_{0}$.

As $t_{0}>\lambda_{n}(\{G\geq a\})$, we have
$$\lambda_{n}(\{G\geq a_{0}\})>\lambda_{n}(\{G\geq a\}),$$
which is a contradiction to
$$a_{0}\geq a$$

Then the following holds:
\begin{equation}
\label{equ:0807a}
\{t|\sup\{a_{1}|\lambda_{n}(\{G\geq a_{1}\})\geq t\}\geq a\}=\{t|\lambda_{n}(\{G\geq a\})\geq t\},
\end{equation}
where $t>0$.

According to the definition of $F_{G}$ and equality \ref{equ:0807a},
it follows that
$$\{t|F_{G}(t)\geq a\}
=\{t|\sup\{a_{1}|\mu(\{G\geq a_{1}\})\geq t\}\geq a\}=\{t|\lambda_{n}(\{G\geq a\})\geq t\},$$
where $t>0$.

Note that $$\mu_{\mathbb{R}}(\{t|F_{G}(t)\geq
a\})=\mu_{\mathbb{R}}\{t|\lambda_{n}(\{G\geq a\})\geq
t\}=\lambda_{n}(\{G\geq a\}),$$ for $t>0$. We have thus proved the
present lemma.
\end{proof}

Denote by $$s(y):=y^{-1}(-\log y)^{-1},$$ where $y\in (0,e^{-1})$.
It is clear that $s$ is strictly decreasing on $(0,e^{-1})$.

We define a function $u$ by
$$u(s(y))=y^{-1},$$
where
$u\in C^{\infty}((e,+\infty))$.
It is clear that $u$ is strictly increasing on $(e,+\infty)$,

The second lemma is about the measure of the level set of $G$:

\begin{Lemma}
\label{l:open_a} For $A>e$, we have
$$\liminf_{A\to+\infty}\lambda_{n}(\{G> A\})u(A)=0,$$
where
$\lambda_{n}$ is the Lebesgue measure of $\mathbb{C}^{n}$.
Especially, $\lim_{A\to+\infty}\frac{A}{u(A)}=0$.
\end{Lemma}

\begin{proof}
According to the definition of Lebesgue integration and Lemma
\ref{l:level}, it follows that
$$\int_{0}^{\mu(\Omega)}F_{G}(t)dt=\int_{\Omega}Gd\lambda_{n}<+\infty.$$
Then we have
$$\liminf_{t\to 0}\frac{F_{G}(t)}{t^{-1}(-\log t)^{-1}}=0,$$
which implies that there exists $t_{j}\to 0$ when $j\to+\infty$,
such that
\begin{equation}
\label{equ:lim_F_G}
\lim_{j\to+\infty}\frac{F_{G}(t_{j})}{t_{j}^{-1}(-\log t_{j})^{-1}}=0.
\end{equation}

Using Lemma \ref{l:level}, we have
$$\lambda_{n}(\{G> F_{G}(t_{j})\})=\mu_{\mathbb{R}}(\{t|F_{G}(t)> F_{G}(t_{j})\})\leq\mu_{\mathbb{R}}((0,t_{j}))=t_{j}.$$

We now want to prove that $u(F_{G}(t_{j}))=o(t_{j}^{-1})$ by
contradiction: if not, there exists $\varepsilon_{0}>0$, such that
$u(F_{G}(t_{j}))\leq \varepsilon_{0}t_{j}^{-1}$. However,
$$u(F_{G}(t_{j}))\leq \varepsilon_{0}t_{j}^{-1}=u(\frac{1}{\frac{t_{j}}{\varepsilon_{0}}(-\log\frac{t_{j}}{\varepsilon_{0}})}).$$

According to the strictly increasing property of $u$, it follows
that $F_{G}(t_{j})\leq
\frac{1}{\frac{t_{j}}{\varepsilon_{0}}(-\log\frac{t_{j}}{\varepsilon_{0}})}$,
which is contradict to equality \ref{equ:lim_F_G} because
$\lim_{t_{j}\to0}\frac{t_{j}(-\log
t_{j})}{\frac{t_{j}}{\varepsilon_{0}}(-\log\frac{t_{j}}{\varepsilon_{0}})}=\varepsilon_{0}$.
Now we obtain $u(F_{G}(t_{j}))=o(t_{j}^{-1})$.

Then we have
$$\lim_{j\to+\infty}\mu(\{G> F_{G}(t_{j})\})u(F_{G}(t_{j}))\leq\lim_{j\to+\infty}t_{j}o(t_{j}^{-1})=0.$$

Note that if $F_{G}(t_{j})$ is bounded above, when $t_{j}$ tends to
$0$, then $G$ has a positive upper bound. Therefore $\mu(\{G>
A\})=0$ for $A$ large enough.

Then we have proved
$$\liminf_{A\to+\infty}\lambda_{n}(\{G> A\})u(A)=0.$$

As $\frac{1}{t(-\log t)}$ is strictly decreasing on $(0,e^{-1})$,
then for any $A>e$, there exists $t_{A}$ such that

1). $\frac{1}{t_{A}(-\log t_{A})}=A$;

2). $t_{A}$ goes to zero, when $A$ goes to $+\infty$.

As
$$\frac{A}{u(A)}=\frac{\frac{1}{t_{A}(-\log t_{A})}}{u(\frac{1}{t_{A}(-\log t_{A})})}=\frac{\frac{1}{t_{A}(-\log t_{A})}}{\frac{1}{t_{A}}}=\frac{1}{-\log t_{A}},$$
then we obtain
$$\lim_{A\to+\infty}\frac{A}{u(A)}=0,$$
by the above property $2)$ of $t_{A}$.

The present lemma is thus proved.

\end{proof}

\subsection{Estimation of integration of holomorphic functions on singular Riemann surfaces}

\begin{Lemma}
\label{l:open_b} Let $h\not\equiv0$ be a holomorphic function on the
disc $\Delta_{r}$ in $\mathbb{C}$. Let $f_{a}$ be a holomorphic
function on $\Delta_{r}$, which satisfies $f_{a}|_{o}=0$ and
$f_{a}(b)=1$ for any $b^{k}=a ($$k$ is a positive integer), then we
have
$$\int_{\Delta_{r}}|f_{a}|^{2}|h|^{2}d\lambda_{1}>C_{1}|a|^{-2},$$
where $a\in\Delta_{r}$ whose norm is small enough, $C_{1}$ is a
positive constant independent of $a$ and $f_{a}$.
\end{Lemma}

\begin{proof}
As $h\not\equiv0$, we may write $h=z^{i}h_{1}$ near $o$, where
$h_{1}|_{o}\neq 0$. Then there exists $r'<r$, such that
$|h_{1}|_{|\Delta_{r'}}\geq C_{0}>0$. Therefore it suffices to
consider the case that $h=z^{i}$ on $\Delta_{r'}$.

By Taylor expansion at $o$,
we have
$$f(z)=\sum_{j=1}^{\infty}c_{j}z^{j}.$$

As $f(b)=1$, then
$$\sum_{j=1}^{\infty}c_{kj}a^{j}=\frac{1}{k}\sum_{1\leq l\leq k}\sum_{j=1}^{\infty}c_{j}b_{l}^{j}=1$$
where $b_{l}^{k}=a$, and $\sum_{1\leq l\leq k} b_{l}^{j}=0$ when $0<j<k$.

It is clear that
\begin{equation}
\label{equ:0808a}
\int_{\Delta_{r'}}|f_{a}|^{2}|h|^{2}d\lambda_{1}=\int_{\Delta_{r'}}|f_{a}|^{2}|z^{i}|^{2}d\lambda_{1}=
2\pi\sum_{j=1}^{\infty}|c_{j}|^{2}\frac{{r'}^{2j+2i+2}}{2j+2i+2}
\end{equation}

Using Schwartz Lemma, we have
\begin{equation}
\label{}
\begin{split}
&(\sum_{j=1}^{\infty}|c_{j}|^{2}\frac{{r'}^{2j+2i+2}}{2j+2i+2})(\sum_{j=1}^{\infty}\frac{2j+2i+2}{{r'}^{2j+2i+2}}|a|^{2j})
\\&\geq(\sum_{j=1}^{\infty}|c_{kj}|^{2}\frac{{r'}^{2kj+2i+2}}{2kj+2i+2})(\sum_{j=1}^{\infty}\frac{2j+2i+2}{{r'}^{2j+2i+2}}|a|^{2j})
\geq|\sum_{j=1}^{\infty}c_{kj}a^{j}|^{2}=1.
\end{split}
\end{equation}

Note that
$$\sum_{j=1}^{\infty}\frac{2j+2i+2}{{r'}^{2j+2i+2}}|a|^{2j}
=|\frac{a}{r'}|^{2}((2i+2)\frac{{r'}^{-2i-2}}{1-|\frac{a}{r'}|^{2}}+2\frac{{r'}^{-2i-2}}{(1-|\frac{a}{r'}|^{2})^{2}}),$$
and
$((2i+2)\frac{{r'}^{-2i-4}}{1-|\frac{a}{r'}|^{2}}+2\frac{{r'}^{-2i-4}}{(1-|\frac{a}{r'}|^{2})^{2}})$
has a uniformly upper bound independent of $a$ when
$|a|<\frac{r'}{2}$. The Lemma thus follows.

\end{proof}

Let's recall the
local parametrization theorem:

\begin{Theorem}
\label{t:para}(see \cite{demailly-book})
Let $\mathscr{J}$ be a prime ideal of $\mathcal{O}_{n}$ and
let $\mathcal{C}'=V(\mathscr{J})$ be an analytic curve at $o$.
Then the ring $\mathcal{O}_{n}/\mathscr{J}$ is a finite integral extension of $\mathcal{O}_{d}$;
let $q$ be the degree of the extension.
There exists a local coordinates
$$(z';z'')=(z_{1};z_{2},\cdots,z_{n}),$$
such that
if $\Delta'_{r'}$ and $\Delta''_{r''}$ are polydisks of sufficient small
radii $r'$ and $r''$ and if $r'\leq \frac{r''}{C}$ with $C$ large,
the projective map $\pi':\mathcal{C}'\cap (\Delta'_{r'}\times\Delta''_{r''})\to\Delta'_{r'}$
is a ramified covering with $q$ sheets, whose ramification locus
is contained in $S=\{o'\}\subset\Delta'_{r'}$.
This means that

a), the open set $\mathcal{C}'_{S}:=\mathcal{C}'\cap ((\Delta'_{r'}\setminus S)\times\Delta''_{r''})$
is a smooth $1-$dimensional manifold, dense in $\mathcal{C}'\cap(\Delta'_{r'}\times\Delta''_{r''})$;

b), $\pi':\mathcal{C}'_{S}\to\Delta'_{r'}\setminus S$ is an unramified covering;

c), the fibre $\pi^{'-1}(z')$ have exactly $q$ elements if $z'\in \Delta'\setminus S$ and at most $q$ if $z'\in S$.

Moreover, $\mathcal{C}'_{S}$ is a connected covering of $\Delta'_{r'}\setminus S$,
and $\mathcal{C}'\cap(\Delta'_{r'}\times\Delta''_{r''})$ is contained in a cone $|z''|\leq\frac{C}{6}|z'|$.
\end{Theorem}

Let $\Delta'$ and $\Delta''$ be unit disk with coordinates $(z_{1})$
and unit polydisc with coordinates $(z_{2},\cdots,z_{n})$
respectively.

Let
$$\pi:\Delta'\times\Delta''\to \Delta'$$
be the projective map which is given by
$$\pi(z_{1};z_{2},\cdots,z_{n})=z_{1}.$$

We have the following Remark of Theorem \ref{t:para}.

\begin{Remark}
\label{r:para}
Let $\mathscr{J}$ be a prime ideal of $\mathcal{O}_{n}$ and
let $\mathcal{C}'=V(\mathscr{J})$ be an analytic curve at $o$.
Then the ring $\mathcal{O}_{n}/\mathscr{J}$ is a finite integral extension of $\mathcal{O}_{d}$;
let $q$ be the degree of the extension,
there exists a biholomorphic map $j$ from
a neighborhood of $\overline{\Delta'\times\Delta''}$
to a neighborhood  $U_{o}$ of $o$,
such that the projective map $\pi|_{\mathcal{C}\cap (\Delta'\times\Delta'')}\to\Delta'$
is a ramified covering with $q$ sheets, whose ramification locus
is contained in $S=\{o'\}\subset\Delta'$
where
$$\mathcal{C}:=j^{-1}(\mathcal{C}').$$
This means that

a), the open set $\mathcal{C}_{S}:=\mathcal{C}\cap ((\Delta'\setminus S)\times\Delta'')$
is a smooth $1-$dimensional manifold, dense in $\mathcal{C}\cap(\Delta'\times\Delta'')$;

b), $\pi|_{\mathcal{C}_{S}}:\mathcal{C}_{S}\to\Delta'\setminus S$ is an unramified covering;

c), the fibre $\pi^{-1}(z')$ has exactly $q$ elements if
$z'\in\Delta'\setminus S$ and at most $q$ if $z'\in S$.

Moreover, $\mathcal{C}_{S}$ is a connected covering of $\Delta'\setminus S$,
and $\mathcal{C}\cap(\Delta'\times\Delta'')$ is contained in a cone $|z''|\leq\frac{1}{6}|z'|$.

\end{Remark}

Using Lemma \ref{l:open_b} and Remark \ref{r:para}, we obtain the
following singular version of Lemma \ref{l:open_b}:

\begin{Lemma}
\label{l:open_sing}
Let $h$ be a holomorphic function on an analytic curve $\mathcal{C}$ as in Remark \ref{r:para}.
Let $f_{a}$ be a holomorphic function on $\mathcal{C}$,
which satisfies
$f(o)=0$ and $f_{a}(\pi^{-1}(a)\cap\mathcal{C})=1$,
then we have
$$\int_{\mathcal{C}_{S}}|f_{a}|^{2}|h|^{2}(\pi|_{\mathcal{C}_{S}})^{*}d\lambda_{\Delta'}>C_{2}|a|^{-2},$$
when $a\in\Delta'$ and $|a|$ is small enough, where $C_{2}$ is a
positive constant independent of $a$ and $f_{a}$.
\end{Lemma}

\begin{proof}
As $(\mathcal{C},o)$ is irreducible and locally irreducible, then
there is a normalization $j_{nor}:(\Delta,0)\to(\mathcal{C},o)$,
denoted by
$$j_{nor}(t)=(g_{1}(t),\cdots,g_{n}(t)),$$ where $t$ is the
coordinate of $\Delta$. As $\pi_{\mathcal{C}}$ is a covering, then
$g_{1}\not\equiv0$.

Without loss of generality, we may assume $g_{1}(t)=t^{i_{1}}$ on
$\Delta_{r_0}$ for small enough $r_{0}\in(0,1)$.

There is a given $r>0$, which is small enough, such that
$$(\mathcal{C}\cap\Delta^{n}_{r_{0}})\supset\{(t^{i_1},g_{2}(t),\cdots,g_{n}(t))|t\in\Delta_{r}\},$$
where $i_{1}\geq 1$, and $g_{i}$ ($i\geq 2$) are holomorphic
functions on $\Delta_{r}$ satisfying $|g_{i}|\leq \frac{1}{6}
|t^{i_1}|$.

For given $r'<r$ small enough, we have
\begin{equation}
\label{equ:0808b}
\begin{split}
\int_{\mathcal{C}_{S}}|f_{a}|^{2}|h|^{2}(\pi|_{\mathcal{C}_{S}})^{*}d\lambda_{\Delta'}
&\geq i_{1}^{2}\int_{\Delta_{r'}}|j_{nor}^{*}f_{a}(t)|^{2}|j_{nor}^{*}h(t)|^{2}|t|^{2(i_{1}-1)}d\lambda_{\Delta}
\\&=i_{1}^{2}\int_{\Delta_{r'}}|j_{nor}^{*}f_{a}(t)(t^{i_{1}-1}j_{nor}^{*}h(t))|^{2}d\lambda_{\Delta},
\end{split}
\end{equation}
for any $a$ satisfying $|a|^{\frac{1}{i_1}}\in\Delta_{r'}$.

As
$$j_{nor}^{*}f_{a}(b)=f_{a}(b^{i_1},g_{2}(b),\cdots,g_{n}(b))$$
and
$$(b^{i_1},g_{2}(b),\cdots,g_{n}(b))\subset(\pi^{-1}(b^{i_1})\cap\mathcal{C}),$$
then we have
$$j_{nor}^{*}f_{a}(b)=1,$$
for any $b^{i_{1}}=a$.

Using Lemma \ref{l:open_b}, we have
$$\int_{\Delta_{r'}}|j_{nor}^{*}f_{a}(t)(t^{i_{1}-1}j_{nor}^{*}h(t))|^{2}d\lambda_{\Delta}\geq C_{1}|a|^{-2},$$
where $C_{1}$ is independent of $a$ and $f_{a}$.

Combining with inequality \ref{equ:0808b}, we thus obtain the
present lemma.
\end{proof}

As $\mathcal{C}\cap(\Delta'\times\Delta'')$ is contained in a cone
$|z''|\leq\frac{1}{6}|z'|$, using the submean value property of
plurisubharmonic function, we obtain the following lemma:

\begin{Lemma}
\label{l:approx.L2}
For any holomorphic function $F$, which is a holomorphic on a $\Delta'\times\Delta''$,
we obtain an approximation of the $L^{2}$ norm of $F$:

$$\int_{\Delta'\times\Delta''}|F|^{2}d\lambda_{n}\geq
C_{3}\int_{\mathcal{C}_{S}}|F|_{\mathcal{C}_{S}}|^{2}(\pi|_{\mathcal{C}_{S}})^{*}d\lambda_{\Delta'},$$
where $C_{3}$ is a positive constant independent of $F$. Here all
symbols $\mathcal{C}_{S}$, $\Delta'$ and $\pi$ are the same as in
Remark \ref{r:para}.
\end{Lemma}

\begin{proof}
Using the Fubini Theorem,
$$\int_{\Delta'\times\Delta''}|F|^{2}d\lambda_{n}=
\int_{\Delta'}(\int_{\{z'\}\times\Delta''}|F|^{2}d\lambda_{n-1})d\lambda_{\Delta'},$$
and the submean value inequality of plurisubharmonic function, we
have
$$\int_{\{z'\}\times\Delta''}|F|^{2}d\lambda_{n-1}\geq(\frac{\pi}{3})^{n-1}|F(z',z'')|^{2},$$
for $|z''|\leq\frac{1}{6}$.

If $w=(z',z'')\in(\pi^{-1}(z')\cap\mathcal{C}_{S})$, then $|z''|\leq
\frac{1}{6}$.

As
$$\int_{\Delta'\setminus \{0\}}(\sum_{w\in(\pi^{-1}(z')\cap\mathcal{C}_{S})}|F(w)|^{2})(z')d\lambda_{\Delta'}
=\int_{\mathcal{C}_{S}}|F|_{\mathcal{C}_{S}}|^{2}(\pi|_{\mathcal{C}_{S}})^{*}d\lambda_{\Delta'},$$
it follows that
\begin{equation}
\label{equ:0922.1}
\begin{split}
q\int_{\Delta'\times\Delta''}|F|^{2}d\lambda_{n}=
&q\int_{\Delta'\setminus \{0\}}(\int_{\{z'\}\times\Delta''}|F|^{2}d\lambda_{n-1})d\lambda_{\Delta'}
\\\geq& (\frac{\pi}{3})^{n-1}\int_{\Delta'\setminus \{0\}}(\sum_{w\in(\pi^{-1}(z')\cap\mathcal{C}_{S})}|F(w)|^{2})(z')d\lambda_{\Delta'}
\\=&(\frac{\pi}{3})^{n-1}\int_{\mathcal{C}_{S}}|F|_{\mathcal{C}_{S}}|^{2}(\pi|_{\mathcal{C}_{S}})^{*}d\lambda_{\Delta'},
\end{split}
\end{equation}
where $q$ is the degree of the covering map
$\pi|_{\mathcal{C}_{S}}$.
\end{proof}

\subsection{$L^{2}$ extension theorem with negligible weight}
$\\$

We state optimal constant version of the Ohsawa's $L^2$ extension
theorem with negligible weight (\cite{ohsawa3}) as follows:

\begin{Theorem}\label{t:guan-zhou12}
\cite{guan-zhou12}
Let $X$ be a Stein manifold of dimension n. Let $\varphi+\psi$ and
$\psi$ be plurisubharmonic functions on $X$. Assume that $w$ is a
holomorphic function on $X$ such that
$\sup\limits_X(\psi+2\log|w|)\leq0$ and $dw$ does not vanish
identically on any branch of $w^{-1}(0)$. Put $H=w^{-1}(0)$ and
$H_0=\{x\in H:dw(x)\neq 0\}$. Then  there exists a uniform constant
$\mathbf{C}=1$ independent of $X$, $\varphi$, $\psi$ and $w$ such
that, for any holomorphic $(n-1)$-form $f$ on $H_0$ satisfying
$$c_{n-1}\int_{H_0}e^{-\varphi-\psi}f\wedge\bar f<\infty,$$
where
$c_{k}=(-1)^{\frac{k(k-1)}{2}}(\sqrt{-1})^k$ for $k\in\mathbb{Z}$,
there exists a holomorphic $n$-form F on $X$ satisfying $F=dw\wedge
\tilde{f}$ on $H_0$ with $\imath^*\tilde{f}=f$ and
$$c_n\int_{X}e^{-\varphi}F\wedge\bar F\leq 2\mathbf{C}\pi
c_{n-1}\int_{H_0}e^{-\varphi-\psi}f\wedge\bar f,$$
where
$\imath:H_0\longrightarrow X$ is the inclusion map.
\end{Theorem}

\subsection{Curve selection lemma and Noetherian property of coherent sheaves}
$\\$

We give the existence of some kind of germs of analytic curves,
which will be used.

\begin{Lemma}
\label{l:curve_exist}
Let $(Y,o)$ be a germ of irreducible analytic subvariety in $\mathbb{C}^{n}$,
and $(A,o)$ be a germ of analytic subvariety of $(Y,o)$,
such that $dim A<dim Y$.
Then there exists a germ of holomorphic curve $(\gamma,o)$,
such that $\gamma\subset Y$, and $\gamma\not\subset A$.
\end{Lemma}

\begin{proof}
Note that $(Y,{o})$ is locally Stein. Then using Cartan's Theorem
$A$, we obtain the lemma.
\end{proof}

Now we recall the curve selection lemma stated as follows:

\begin{Lemma}
\label{l:curve}(see \cite{demailly2010}) Let $f$,
$g_{1},\cdots,g_{r}\in\mathcal {O}_{n}$ be germs of holomorphic
functions vanishing at $0$. Then we have $|f|\leq C|g|$ for some
constant $C$ if and only if for every germ of analytic curve
$\gamma$ through $0$ there exists a constant $C_{\gamma}$ such that
$|f\circ \gamma|\leq  C_{\gamma}|g\circ \gamma|$.
\end{Lemma}

In order to obtain some uniform properties of $\gamma$,
we need to consider the following Lemma
which was contained in the proof of Lemma \ref{l:curve} in \cite{demailly2010}.

\begin{Lemma}
\label{p:curve}(see \cite{demailly2010}) Let $f$,
$g_{1},\cdots,g_{s}\in\mathcal {O}_{n}$ be germs of holomorphic
functions vanishing at $o$. Assume that for any given neighborhood
of $o$, $|f|\leq C|g|$ doesn't hold for any constant $C$, where
$g=(g_1,...,g_s)$. Then there exists a germ of analytic curve
$\gamma$ through $o$ satisfying $\gamma\cap\{f=0\}=o$, such that
$\frac{g_{i}}{f}|_{\gamma}$ is holomorphic on $\gamma\setminus o$
with
$$\widetilde{\frac{g_{i}}{f}}|_{\gamma}(o)=0,$$
for any $i\in\{1,\cdots,s\}$, where $\widetilde{\frac{g_{i}}{f}}$ is
the holomorphic extension of $\frac{g_{i}}{f}$ from $\gamma\setminus
o$ to $\gamma$.
\end{Lemma}

\begin{proof}
There exists $\Delta_{r}^{n}$, such that
$g_{1},\cdots,g_{s},f\in\mathcal{O}(\Delta_{r}^{n})$. We define a
germ of analytic set $(Y,o)\subset
(\Delta^{n}_{r}\times\mathbb{C}^{s},o)$ by
$$g_{j}(z)=f(z)z_{n+j},\quad 1\leq j\leq s.$$

Let $p$ be a projection
$$p:\Delta^{n}_{r}\times\mathbb{C}^{s}\to\Delta^{n}_{r},$$
such that
$$p((z_{1},\cdots,z_{n}),(z_{n+1},\cdots,z_{n+s}))=(z_{1},\cdots,z_{n}).$$
Then $Y\cap p^{-1}(\Delta^{n}_{r}\setminus \{f=0\})$ is
biholomorphic to $\Delta^{n}_{r}\setminus \{f=0\}$, which is
irreducible. As every analytic variety has an irreducible
decomposition, then $Y$ contains an irreducible component $Y_{f}$
which contains $Y\cap p^{-1}(\Delta^{n}_{r}\setminus \{f=0\})$.

Since $Y_{f}$ is closed, then
$$Y_{f}=\overline{Y\cap p^{-1}(\Delta^{n}_{r}\setminus \{f=0\})}.$$

By assumption, for any given neighborhood of $o$, $|f|\leq C|g|$
doesn't hold for any constant $C$, then there exist a sequence of
positive numbers $C_{\nu}$ which goes to $+\infty$ as
$\nu\to\infty$, and a sequence of points $\{z_{\nu}\}$ in
$\Delta^{n}_{r}$ which is convergent to $o$ as $\nu\to\infty$, such
that $|f(z_{\nu})|> C_{\nu}|g(z_{\nu})|$.

Then $(z_{\nu},\frac{g(z_{\nu})}{f(z_{\nu})})$ converges to $o$ as
$\nu$ tends to $+\infty$, with $f(z_{\nu})\neq 0$.

As $(z_{\nu},\frac{g(z_{\nu})}{f(z_{\nu})})\in Y_{f}$, then $Y_{f}$
contains $o$.

It follows from Lemma \ref{l:curve_exist} that there exists a germ
of analytic curve $(\gamma,o)\subset Y_{f}$ through $o$ satisfying
$\gamma\cap\{f=0\}=o$, such that $\frac{g_{i}}{f}|_{\gamma}$ is
holomorphic on $\gamma\setminus o$ for each $i\in \{1,\cdots,s\}$.

By the Riemann removable singularity theorem, it follows that
$\frac{g_{i}}{f}|_{\gamma\setminus o}$ can be extended to $\gamma$,
and
$$\widetilde{\frac{g_{i}}{f}}|_{\gamma}(o)=0,$$
for any $i\in\{1,\cdots,s\}$.
\end{proof}

\begin{Remark}
\label{r:curve927} Let $g_{1},\cdots,g_{s}\in\mathcal {O}_{n}$ be
germs of holomorphic functions vanishing at $o$, and $f(o)\neq 0$.
Then there exists a germ of analytic curve $\gamma$ through $o$
satisfying $\gamma\cap\{f=0\}=\emptyset$, such that
$\frac{g_{i}}{f}|_{\gamma}$ is holomorphic on $\gamma$ with
$\frac{g_{i}}{f}(o)=0,$ for any $i\in\{1,\cdots,s\}$.
\end{Remark}

Let's recall a strong Noetherian property of coherent sheaves as
follows:

\begin{Lemma}
\label{l:strong_noeth}(see \cite{demailly-book}) Let $\mathscr{F}$
be a coherent analytic sheaf on a complex manifold $M$, and let
$\mathscr{F}_{1}\subset\mathscr{F}_{2}\subset\cdots$ be an
increasing sequence of coherent subsheaves of $\mathscr{F}$. Then
the sequence $(\mathscr{F}_{k})$ is stationary on every compact
subset of $M$.
\end{Lemma}

\begin{Remark}
\label{r:station}
By Lemma \ref{l:strong_noeth},
it is clear that
$\cup_{\varepsilon>0}\mathcal{I}((1+\varepsilon)\varphi)$ is a
coherent subsheaf of $\mathcal{I}(\varphi)$;
actually for any open $V_{1}\subset\subset M$,
there exists $\varepsilon_{1}>0$,
such that $\cup_{\varepsilon>0}\mathcal{I}((1+\varepsilon)\varphi)|_{V_{1}}=\mathcal{I}((1+\varepsilon_{1})\varphi)$
\end{Remark}

By Remark \ref{r:station},
we derive the following proposition about the generators of the coherent sheaf
$\mathcal{I}_{+}(\varphi)$:

\begin{Proposition}
\label{r:curve} Assume that $F\in\mathcal {O}_{n}$ is a holomorphic
function on a neighborhood $V_{0}$ of $o$, which is not a germ of
$\mathcal{I}_{+}(\varphi)_{o}$. Let
$g_{1},\cdots,g_{s}\in\mathcal{I}_{+}(\varphi)_{0}$ be germs of
holomorphic functions on a neighborhood $V_{1}\subset\subset V_{0}$
of $o$, such that $g_{1},\cdots,g_{s}$ generate
$\mathcal{I}_{+}(\varphi)|_{V_1}$. Then there exists a germ of
analytic curve $\gamma$ on $V_{0}$ through $o$ satisfying
$\gamma\cap\{F=0\}=o$, such that
$\widetilde{\frac{g_{i}\circ\gamma}{F\circ\gamma}}$ is holomorphic
on $\gamma$ for any $i$, and
$$\widetilde{\frac{g_{i}\circ\gamma}{F\circ\gamma}}|_{o}=0,$$
where $\widetilde{\frac{g_{i}\circ\gamma}{F\circ\gamma}}$ is the
holomorphic extension of $\frac{g_{i}\circ\gamma}{F\circ\gamma}$
from $\gamma\setminus o$ to $\gamma$. Moveover, for any germ $g$ of
$\mathcal{I}((1+\varepsilon)\varphi)_{0}$,
$\widetilde{\frac{g\circ\gamma}{F\circ\gamma}}$ is holomorphic on
$\gamma\setminus o$, and
$$\widetilde{\frac{g\circ\gamma}{F\circ\gamma}}|_{o}=0,$$
where $\widetilde{\frac{g\circ\gamma}{F\circ\gamma}}$ is the
holomorphic extension of $\frac{g\circ\gamma}{F\circ\gamma}$ from
$\gamma\setminus o$ to $\gamma$.
\end{Proposition}

\begin{proof}
By Remark \ref{r:station}, there exists $\varepsilon_{1}>0$, such
that
$g_{1},\cdots,g_{s}\in\mathcal{I}((1+\varepsilon_{1})\varphi)(V_{1})$.
As $F$ is not a germ of
$\mathcal{I}_{+}(\varphi)_{o}=\mathcal{I}((1+\varepsilon_{1})\varphi)_{o}$,
then for any neighborhood of $o$, $|F|\leq C (\sum_{1\leq j\leq
s}|g_{j}|^{2})^{1/2}$ doesn't hold for any constant $C$.

By Lemma \ref{p:curve} and Remark \ref{r:curve927}, there exists a
germ of analytic curve $\gamma$ on $V_{0}$ through $o$ satisfying
$\gamma\cap\{F=0\}=o$, such that
$\widetilde{\frac{g_{i}\circ\gamma}{F\circ\gamma}}$ is holomorphic
on $\gamma$ for any $i$, and
$$\widetilde{\frac{g_{i}\circ\gamma}{F\circ\gamma}}|_{0}=0,$$
where $\widetilde{\frac{g_{i}\circ\gamma}{F\circ\gamma}}$ is the holomorphic extension of
$\frac{g_{i}\circ\gamma}{F\circ\gamma}$ from $\gamma\setminus o$ to $\gamma$.
\end{proof}

\section{Proof Theorem \ref{t:strong_open}}

We will prove Theorem \ref{t:strong_open} by the methods of
induction and contradiction, and by using dynamically $L^{2}$
extension theorem with negligible weight.

\subsection{Step 1: Theorem \ref{t:strong_open} for dimension 1 case}
$\\$

We first consider Theorem \ref{t:strong_open} for dimension 1 case,
which is elementary but revealing.

We choose $r_{0}$ small enough, such that
$\{F=0\}\cap\Delta_{r_0}\subset \{o\}$.

As $\int_{\Delta}|F|^{2}e^{-\varphi}d\lambda_{1}<+\infty$,
by Lemma \ref{l:open_a},
we have
$$\liminf_{A\to+\infty}\mu(\{|F|^{2}e^{-\varphi}> A\})u(A)=0.$$

It is clear that, for any given $B>0$, there exist $A>B$ and
$z_{A}\in\Delta_{u(A)^{-1/2}}$, such that
$e^{-\varphi(z_{A})}|F(z_{A})|^{2}\leq A$. We assume that $A>10$.

Let $\psi=-\log2$,
then $\log|z'-z_{A}|+\psi<0$.

Using Theorem \ref{t:guan-zhou12} on $\Delta$, we obtain a
holomorphic function $F_{A}$ on $\Delta$ for each $A$ and $p_{A}>1$,
such that $F_{A}|_{z_{A}}=F(z_{A})$, and
\begin{equation}
\label{equ:0927a}
\int_{\Delta}|F_{A}|^{2}e^{-p_{A}\varphi}d\lambda_{1}<8\pi A.
\end{equation}

By the negativeness of $\varphi$, it follows that
\begin{equation}
\label{equ:0927b}\int_{\Delta}|F_{A}|^{2}d\lambda_{1}<8\pi A.
\end{equation}

Assume Theorem \ref{t:strong_open} for $n=1$ is not true. Therefore
$$\int_{\Delta_{r}}|F|^{2}e^{-p\varphi}d\lambda=+\infty,$$
for any $r>0$ and $p>1$.

Since $\{F=0\}\cap\Delta_{r_0}\subset \{o\}$, then it follows from
inequality \ref{equ:0927a} that there exists a holomorphic function
$h_{A}$ on $\Delta_{r_0}$, such that

1). $F_{A}|_{\Delta_{r_0}}=F|_{\Delta_{r_0}}h_{A}$;

2). $h_{A}(o)=0$;

3). $h_{A}(z_{A})=1$.

By Lemma \ref{l:open_b}, it follows that
$$\int_{\Delta_{r_0}}|F_{A}|^{2}d\lambda_{1}>C_{1}|z_A|^{-2}>C_{1}u(A),$$
where $C_{1}$ is independent of $A$.

It contradicts to
$$\int_{\Delta}|F_{A}|^{2}d\lambda_{1}<8\pi A.$$

We have thus proved Theorem \ref{t:strong_open} for $n=1$.

\subsection{Step 2:  Theorem \ref{t:strong_open} for $n=k$}
$\\$

Assume Theorem \ref{t:strong_open} for $n=k$ is not true. Therefore,
$$\int_{\Delta^{k}_{r}}|F|^{2}e^{-\varphi}d\lambda_{k}<+\infty,$$
for some $r>0$,
and
$$\int_{\Delta^{k}_{r}}|F|^{2}e^{-p\varphi}d\lambda_{k}=+\infty,$$
for any $r>0$ and any $p>1$.

Then the germ of the holomorphic function $F$ is in
$\mathcal{I}(\varphi)_{o}$, but is not in
$\mathcal{I}_{+}(\varphi)_{o}$.

Using Proposition \ref{r:curve}, we have a germ of an analytic curve
$\gamma$ through $o$ satisfying $\{F|_{\gamma}=0\}\subseteq \{o\}$,
such that for any germ of the holomorphic function $g$ in
$\mathcal{I}_{+}(\varphi)$, and we also have a holomorphic function
$h_{g}$ on $\gamma$ satisfying
$$h_{g}|_{o}=0,$$
such that
\begin{equation}
\label{equ:infact}
g|_{\gamma}=F|_{\gamma}h_{g}.
\end{equation}

Then we can choose a biholomorphic map $\imath$ from a neighborhood
of $\overline{\Delta'\times\Delta''}$ to a neighborhood
$V_{o}\subset \Delta^{k}$ of $o$, which is small enough, with origin
keeping $\imath(o)=o$, such that

1), $\imath^{-1}(\gamma)$ is a closed analytic curve in the
neighborhood of $\overline{\Delta'\times\Delta''}$;

2), $\imath^{-1}(\gamma)$ satisfies the parametrization property as
the analytic curve $\mathcal{C}$ in Remark \ref{r:para}.

Note that
$$\int_{V_{o}}|F|^{2}e^{-p\varphi}d\lambda_{n}=
\int_{\Delta'\times\Delta''}|\imath^{*}(F)|^{2}e^{-p\imath^{*}(\varphi)}\imath^{*}(d\lambda_{n}).$$
Then
$$\int_{\Delta^{k}_{r}}|F|^{2}e^{-p\varphi}d\lambda_{n}=+\infty,$$
for any $r>0$ and any $p>1$, is equivalent to
$$\int_{\Delta'_{r}\times\Delta''_{r}}|\imath^{*}(F)|^{2}e^{-p\imath^{*}(\varphi)}d\lambda_{n}=+\infty,$$
for any $r>0$ and any $p>1$.

As
$\int_{\Delta'\times\Delta''}|\imath^{*}(F)|^{2}e^{-\imath^{*}(\varphi)}d\lambda_{k}<+\infty$,
it follows from Lemma \ref{l:open_a} that
$$\liminf_{A\to+\infty}\mu(\{z_{1}|\int_{\pi^{-1}(z_{1})}|\imath^{*}(F)|^{2}e^{-\imath^{*}(\varphi)}d\lambda_{k-1}> A\})u(A)=0,$$
for any $j$, where $\pi$ is the projection in Remark \ref{r:para}.

It is clear that for any given $B>0$, there exists $A>B$, such that
$$\{z_{k}|\int_{\pi^{-1}(z_{1})}|\imath^{*}(F)|^{2}e^{-\imath^{*}(\varphi)}d\lambda_{k-1}> A\}$$
cannot contain
$\Delta_{u(A)^{-1/2}}$.

Then for any given $B>0$, there exist $A>B$ and
$z_{A}\in\Delta_{u(A)^{-1/2}}$, such that
$$\int_{\pi^{-1}(z_{A})}|\imath^{*}(F)|^{2}e^{-\imath^{*}(\varphi)}d\lambda_{k-1}\leq A.$$
We assume that $A>e^{10}$.

\subsubsection{Using dynamically $L^2$ extension theorem with negligible weight}
$\\$

As the conjecture for $n=k-1$ holds, there exists a positive number
$p_{A}>1$, such that
$$\int_{\pi^{-1}(z_{A})}|\imath^{*}(F)|^{2}e^{-p_{A}\imath^{*}(\varphi)}d\lambda_{k-1}<2A.$$

Let $\psi=-\log2$, then $\log|z'-z_{A}|+\psi<0$. Using Theorem
\ref{t:guan-zhou12} on $\Delta'\times\Delta''$, we obtain a
holomorphic function $F_{A}$ on $\Delta'\times\Delta''$ for each
$A$, such that
$F_{A}|_{\pi^{-1}(z_{A})}=\imath^{*}(F)|_{\pi^{-1}(z_{A})}$, and
$$\int_{\Delta'\times\Delta''}|F_{A}|^{2}e^{-p_{A}\imath^{*}(\varphi)}d\lambda_{k}<8\pi A.$$

It follows form equality \ref{equ:infact} that there exists a
holomorphic function on $\gamma$, denoted by $h_{A}$, such that
\begin{equation}
\label{equ:infact1}
\imath_{*}F_{A}|_{\gamma}=F|_{\gamma}h_{A},
\end{equation}
therefore,
\begin{equation}
\label{equ:infact1}
F_{A}|_{\imath^{-1}(\gamma)}=\imath^{*}(F)|_{\imath^{-1}(\gamma)}\imath^{*}(h_{A}),
\end{equation}
where $h_{A}(0)=0$, $\imath^{*}(h_{A})(\imath^{-1}(\gamma)\cap \pi^{-1}(z_{A}))=1$.

It follows from the negativeness of $\varphi$ that
$$\int_{\Delta'\times\Delta''}|F_{A}|^{2}d\lambda_{1}<8\pi A.$$

Using equality \ref{equ:infact1}, the condition
$|z_{A}|<u(A)^{-\frac{1}{2}}$, and Lemma \ref{l:open_sing}, we have
$$\int_{\imath^{-1}(\gamma)}|F_{A}|^{2}d\pi^{*}\lambda_{\Delta'}>C_{1}u(A),$$
where $C_{1}>0$ is independent of $A$ and $F_{A}$. In our use of
Lemma \ref{l:open_sing}, the function $h_A$ corresponds to $f_a$
(where $a=z_A$) in Lemma \ref{l:open_sing}, which does not
correspond to $h$ in the Lemma (actually $F|_{\gamma}$ corresponds
to $h$).

Using Lemma \ref{l:approx.L2}, we obtain
$$\int_{\Delta'\times\Delta''}|F_{A}|^{2}d\lambda_{k}\geq C_{3}\int_{\imath^{-1}(\gamma)}|F_{A}|^{2}d\pi^{*}\lambda_{\Delta'},$$
where $C_{3}>0$ is independent of $A$ and $F_{A}$.

Therefore
$$\int_{\Delta'\times\Delta''}|F_{A}|^{2}d\lambda_{k}\geq C_{1} C_{3}u(A),$$
which contradicts to
$$\int_{\Delta'\times\Delta''}|F_{A}|^{2}d\lambda_{k}<8\pi A,$$
for $A$ large enough.

We have thus proved Theorem \ref{t:strong_open} for $n=k$.

The proof of Theorem \ref{t:strong_open} is thus complete.

%%%------------------------------------------------------------------------

\bibliographystyle{references}
\bibliography{xbib}

\end{document}